\documentclass[12pt]{amsart}

\usepackage{amsfonts}
\usepackage{amssymb}
\usepackage{amsmath}

\usepackage{epsf}
\usepackage{epsfig}
\usepackage{colordvi}
\usepackage{fancybox}

\newcommand{\RR}{\mathbb R}
\newcommand{\CC}{\mathbb C}

\newcommand{\DD}{\mathbb D}

\newcommand{\jj}{\mathrm{j}}
\newcommand{\qq}{\mathrm{q}}

\newcommand{\manifold}[1]{\mathcal{#1}}
\newcommand{\M}{\manifold{M}}
\newcommand{\D}{\manifold{D}}

\newcommand{\vect}[1]{\mathrm{#1}} 
\newcommand{\x}{\vect{x}}
\newcommand{\y}{\vect{y}}
\newcommand{\va}{\vect{a}}
\newcommand{\vb}{\vect{b}}

\newcommand{\vX}{\vect{X}}

\newcommand{\vH}{\vect{H}}
\newcommand{\n}{\vect{n}}

\newtheorem{thm}{Theorem}[section]
\newtheorem{prop}[thm]{Proposition}

\theoremstyle{remark}
\newtheorem{remark}{Remark}[section]
\newtheorem{example}{Example}[section]
\theoremstyle{definition}
\newtheorem{dfn}{Definition}[section]

\newcommand{\ds}{\displaystyle}

\begin{document}

\title[Explicit Solving of the Natural PDEs System of  Minimal Lorentz Surfaces in $\mathbb R^4_2$]
{Explicit Solving of the System of Natural PDEs of Minimal Lorentz Surfaces in  $\mathbb R^4_2$}

\author{Krasimir Kanchev}
\author{Ognian Kassabov}
\author{Velichka Milousheva}

\address {Department of Mathematics and Informatics, Todor Kableshkov University of Transport,
158 Geo Milev Str., 1574, Sofia, Bulgaria}%
\email{kbkanchev@yahoo.com}%

\address{Institute of Mathematics and Informatics, Bulgarian Academy of Sciences,
Acad. G. Bonchev Str. bl. 8, 1113, Sofia, Bulgaria}
\email{okassabov@math.bas.bg}

\address{Institute of Mathematics and Informatics, Bulgarian Academy of Sciences,
Acad. G. Bonchev Str. bl. 8, 1113, Sofia, Bulgaria}
\email{vmil@math.bas.bg}



\begin{abstract}
A minimal Lorentz surface in $\RR^4_2$ is said to be of general type if its corresponding null curves are non-degenerate.
These surfaces admit canonical isothermal and canonical isotropic coordinates.
It is known that the Gauss curvature $K$ and the normal curvature $\varkappa$ of such a surface 
considered as functions of the canonical coordinates satisfy a system of two natural PDEs.
Using the Weierstrass type representations of the corresponding null curves,
we solve explicitly the system of natural PDEs, expressing any solution by means of four real functions of one variable.
We obtain the transformation formulas for the functions in the Weierstrass representation of a null curve under a proper motion in $\RR^4_2$.
Using this, we find the relation between two quadruples of real functions generating one and the same solution to the system of natural PDEs.
\end{abstract}

\maketitle


\dedicatory{\textit{Dedicated to the  memory of our Teacher Prof. Georgi Ganchev (1945 - 2020)}}

\section{Introduction} \label{S:Intro}

In the present paper we study minimal Lorentz surfaces of \emph{general type} in the pseudo-Euclidean 4-space 
$\RR^4_2$ with neutral metric. These are  surfaces with  indefinite metric in $\RR^4_2$  whose mean curvature vector field  $\vH$ is zero and whose Gauss curvature  $K$  and normal curvature (curvature of the normal connection) $\varkappa$ satisfy the condition 
 $K^2-\varkappa^2\neq 0$. 

 In \cite{Sakaki2011}, Sakaki derived a system of natural PDEs for the curvatures $K$ and $\varkappa$ of this class of surfaces. 
It is proved that under the additional condition $K^2-\varkappa^2>0$ there exist exactly two one-parameter families of minimal Lorentz surfaces determined by a given solution  $(K,\varkappa)$ to the system of natural PDEs written in terms of  isothermal coordinates. 

Special isothermal coordinates called \emph{canonical} are introduced in \cite{A-M-1} for the class of minimal Lorentz surfaces satisfying $K^2-\varkappa^2>0$. It is proved that  the solution  $(K,\varkappa)$  to the system of PDEs expressed in terms of canonical coordinates determines uniquely the geometry of the surface. 

The general case $K^2-\varkappa^2\neq 0$ of minimal Lorentz surfaces in $\RR^4_2$ is studied in \cite{Kanchev2020} where these surfaces are divided into three types. Canonical coordinates are introduced for each type of surfaces and it is proved that in each case the geometry of the surface is determined uniquely by the solution $(K,\varkappa)$ to the system of natural PDEs expressed in terms of canonical coordinates. 

 A representation of a minimal Lorentz surface was given by M.P. Dussan and M. Magid in \cite{D-Mag} where they solved the Bj\"orling problem for timelike surfaces in $\RR^4_2$ constructing  a special  normal frame and a split-complex representation formula. The Bj\"orling problem for timelike surfaces in the Lorentz-Minkowski spaces $\RR^3_1$ and $\RR^4_1$ is solved in \cite{Ch-D-Mag} and  \cite{D-Fil-Mag}, respectively. Spinor representation of Lorentz surfaces in the pseudo-Euclidean 4-space with neutral metric is given in \cite{Bay-Patty}.
In \cite{Patty}, V. Patty gave a generalized Weierstrass representation of a minimal Lorentz surface in $\RR^4_2$ using spinors and Lorentz numbers (also known as para-complex, split-complex, double or hyperbolic numbers),  thus extending the Weierstrass representation of a minimal surface in $\RR^3_1$ given by J. Konderak \cite{Kond}. 

A basic instrument in the study of  minimal surfaces in the Euclidean space is the generalized Gauss map (see \cite{H-O-1}).
In \cite{Kanchev2020}, the generalized Gauss map $\Phi$ of a minimal Lorentz surface in $\RR^4_2$  considered as a holomorphic function over the algebra of the double numbers  has been studied. Another  method for studying such a surface is by 
use of  a pair of null curves in  $\RR^4_2$ \cite{Chen-1}.
In Section \ref{Can_Coord-MinLorSurf_R42}, we reveal the relation between the generalized Gauss map $\Phi$, the classification of minimal surfaces and their canonical coordinates on  one hand, and their corresponding null curves on the other. 
In Section \ref{K_kappa-IsoCurvs_R42}, we give the relation  between the  basic invariants  $K$ and $\varkappa$  of a minimal Lorentz surface and its corresponding null curves.

The next problem  that arises in the  study of the system of natural equations describing the minimal surfaces,
is the problem of finding its explicit solutions.
This problem is solved  for the classes of spacelike minimal surfaces in $\RR^4$, $\RR^4_1$, and  $\RR^4_2$, in  papers 
\cite{Kanchev2014}, \cite{Kanchev2017}, and \cite{Kanchev2019}, respectively. The main idea for obtaining explicit solutions is to use a special type of Weierstrass representation in terms of canonical coordinates, which is called \textit{canonical Weierstrass representation}.

The problem of finding explicit solutions in the case of a minimal Lorentz surface in $\RR^4_2$ is studied in \cite{K-M-1}. 
Using the canonical Weierstrass representation for the corresponding null curves, formulas of type \eqref{K_kappa-gh_MinLorSurf_R42} are obtained for the Gauss curvature $K$ and the normal curvature $\varkappa$ in the case $K^2-\varkappa^2>0$ and $\varkappa<0$.  Under the last conditions, the formulas give solutions to the considered system, expressed in terms of  four real functions of one variable.
Moreover, the following two open questions arise: \textit{Are all solutions to the system of natural PDEs obtained in the described way?}
\textit{When do two different quadruples of  real functions give one and the same solution?} 

In Sections  \ref{W-IsoCurvs_R42} and \ref{W-MinLorSurf_R42} of the present paper, the method used in \cite{K-M-1} is refined so that the corresponding formulas for the Gauss curvature  and the curvature of the normal connection are obtained in the general case $K^2-\varkappa^2\neq 0$.
The transformation formulas for the Weierstrass  representation under a proper motion  in $\RR^4_2$ are also derived.
Examples of the three types of minimal surfaces parametrized by canonical coordinates are also given.

As an application of the developed theory, in Section \ref{sect-Sol_NatEq_MinLorSurf_R42}, we obtain the general solution to the system of natural PDEs of minimal Lorentz surfaces in $\RR^4_2$.
Theorems \ref{Sol_NatEq_MinLorSurf_R42} and  \ref{Nat_eq_same_K_kappa_hatgh_gh_R42} give a comprehensive answer to the two questions raised above. 
We also give examples of solutions to the system of natural PDEs.


\vspace{0.5ex}
\setcounter{equation}{0}
\section{Preliminaries} \label{S:Prelim}

Let $\RR^4_2$ be the standard four-dimensional pseudo-Euclidean space with neutral metric. The  indefinite inner scalar product in $\RR^4_2$ is given by the formula
\begin{equation*}
\langle \va ,\vb \rangle = -a_1b_1+a_2b_2-a_3b_3+a_4b_4.
\end{equation*}
Let $\M=(\D ,\x)$ be a two-dimensional Lorentz surface in $\RR^4_2$, where $\D\subset\RR^2$, and $\x:\D \to \RR^4_2$ is an immersion. It is well known that each point
 $p \in \M$ has a neighborhood in which isothermal coordinates $(u,v)$ can be introduced  such that the first fundamental form of $\M$  is expressed as follows \cite{Anciaux-1}:
\[
\mathbf{I}=E\,(du^2-dv^2).
\]
Thus, the following formula holds true (see \cite{Anciaux-1}):
\begin{equation*}
\Delta^h \x = 2E\vH,
\end{equation*}
where $\Delta^h$ denotes the hyperbolic Laplace operator, defined by
$\Delta^h = \ds \frac{\partial^2}{\partial u^2} - \frac{\partial^2}{\partial v^2}$, and $\vH$ is the mean curvature vector field  of $\M$.

Hence, a Lorentz surface  $\M=(\D ,\x)$ parametrized by isothermal coordinates is \textit{minimal} 
 ($\vH=0$), if and only if the vector function  $\x$ is \emph{hyperbolic harmonic} ($\Delta^h \x = 0$).

Each hyperbolic harmonic function has the following form 
\begin{equation} \label{MinSurf-IsoCurves}
	\x(u,v) = \frac{\alpha_1(u+v)+\alpha_2(u-v)}{2},
\end{equation}
where $\alpha_1$ and  $\alpha_2$ are vector functions of one real variable in $\RR^4_2$  which are determined uniquely up to an additive constant. The condition on the coordinates to be isothermal is equivalent to 
\begin{equation} \label{IsoCurves-cond}
	\alpha'^2_1 = \alpha'^2_2 = 0; \qquad E = \frac{1}{2} \langle\alpha'_1 , \alpha'_2 \rangle \neq 0.
\end{equation}
Consequently, each minimal Lorentz surface in $\RR^4_2$ corresponds to a pair of null curves $(\alpha_1, \alpha_2)$ satisfying the condition 
$\langle \alpha'_1 , \alpha'_2 \rangle \neq 0$ (see \cite{Chen-1}). 
Conversely, any such pair of curves generates according to  \eqref{MinSurf-IsoCurves} a minimal Lorentz surface $\M$ parametrized by isothermal coordinates. 
The pair $(\alpha_1, \alpha_2)$ is determined uniquely by the surface  $\M$ up to numeration, parametrization and translation.

If $(t_1,t_2)$ is another pair of local coordinates of $\M$ such that 
\begin{equation} \label{isoterm-isotrop}
	t_1 = u+v\,; \qquad t_2 = u-v; \qquad  u = \frac{t_1+t_2}{2}; \qquad  v = \frac{t_1-t_2}{2},
\end{equation}
then \eqref{MinSurf-IsoCurves} implies  $\x_{t_1}=\frac{\alpha'_1}{2}$ and $\x_{t_2}=\frac{\alpha'_1}{2}$. 
Hence, it follows from \eqref{IsoCurves-cond} that  $\x_{t_1} ^2 =0$ and $\x_{t_2} ^2= 0$. 
We will call  such coordinates \textit{isotropic coordinates} of $\M$. Formulas \eqref{isoterm-isotrop}  give the  correspondence between 
isothermal and isotropic coordinates of $\M$.

\smallskip
Let $\alpha = \alpha(t)$ be a null  curve and $t=t(s)$ be a change of the parameter such that $t' \neq 0$. Then, 
\begin{equation}\label{t-s-IsoCurvs}
{\alpha''_s}^2 = {\alpha''_t}^2 {t'}^4,
\end{equation}
which implies that the condition ${\alpha''}^2\neq 0$  does not depend on the parametrization of the curve.
We will briefly call null curves with this property \textit{non-degenerate}. 
A non-degenerate null curve $\alpha $ is said to be parametrized by  a \textit{natural parameter} if  ${\alpha''}^2=\pm 1$. 
Such parameter is also known in the literature as \textit{pseudo arc-length parameter}  \cite{Vessiot1905,Duggal}, since it  plays a role
similar to the role of an arc-length parameter for non-null curves. 
It follows from \eqref{t-s-IsoCurvs} that if $t$ is an arbitrary parameter of  $\alpha$, then a natural parameter $s$ is given by 
\begin{equation}\label{NatParam_IsoCurv}
s = \int\sqrt[4]{\big|{\alpha''}^2 (t)\big|}\:dt\,.
\end{equation}

\medskip
In the study of Lorentz surfaces it is convenient to introduce the algebra of double numbers $\DD$ determined in the following way:
$\DD=\{t=u+\jj v  : \  u,v \in \RR ,\ \jj^2=1 \}$,  $\jj\notin\RR$,  $\jj$ commutes with the elements of $\RR$.
For the element  $t=u+\jj v$ of $\DD$ we have $|t|^2=t\bar t = (u+\jj v)(u-\jj v) = u^2-v^2$.
This shows that $\DD$ is the hyperbolic analogue of the algebra of complex numbers $\CC$ and reflects the Lorentz geometry of $\RR^2_1$.
The algebra of double numbers is used essentially in paper \cite{Kanchev2020} for studying the local properties of Lorentz surfaces in $\RR^4_2$.
We will follow the basic notations and definitions used in \cite{Kanchev2020}.

 In many cases, during computations with  double numbers, it is more convenient along with the basis $(1,\jj)$
 to use the null basis $(\qq,\bar\qq)$, which is determined as follows:
\begin{equation}\label{qq-def}
\qq = \frac{1-\jj}{2}; \qquad \bar\qq = \frac{1+\jj}{2}; \qquad 1 = \bar\qq + \qq; \qquad \jj = \bar\qq - \qq.
\end{equation}
It is easily  seen that the following equalities are valid:
\begin{equation*}
\qq^2 = \qq; \qquad \bar\qq^2 = \bar\qq; \qquad \qq \bar\qq = 0,
\end{equation*} 
which imply that the addition and the multiplication with respect
to the basis $(\qq,\bar\qq)$  are carried out component-wise.
This means that $\DD$, as an algebra, is isomorphic to two copies of $\RR$: \, $\DD=\RR\oplus\RR$.

\smallskip
Foundations of analysis in the algebra of double numbers $\DD$  can be found in \cite{A_F-1}, \cite{DT-K-L-1}, \cite{M-R-1}.

\medskip

Let $\M=(\D ,\x)$ be a minimal Lorentz surface parametrized by isothermal coordinates  $(u,v)$. We introduce the  $\DD^4_2$-valued vector function  $\Psi(t)$, $t=u+\jj v$ in the following way:
\begin{equation*} 
	\Psi(t) = \x(u,v)+\jj \y(u,v),
\end{equation*}
where $\y$  is a hyperbolic harmonic function conjugate to $\x$ and defined by 
\begin{equation*} 
	\y(u,v) = \frac{\alpha_1(u+v)-\alpha_2(u-v)}{2}.
\end{equation*}
Obviously,  $\Psi$ is  a holomorphic function  over $\DD$ $\left(\frac{\partial\Psi}{\partial \bar t}=0\right)$.
Its derivative  $\Phi$ is called the \textit{generalized Gauss map}. We have the following equalities: 
\begin{equation*} \label{Phi-def}
	\Phi = \Psi' = \frac{\partial\Psi}{\partial u} = \x_u +\jj\y_u = \x_u +\jj\x_v,
\end{equation*}
which imply 
\begin{equation} \label{Phi-E}
	\Phi^2 = 0;  \qquad  \|\Phi\|^2 = \Phi \bar\Phi =  \x^2_u - \x^2_v = 2E.
\end{equation}

Let $\sigma$ be the second fundamental form of $\M$. For a given vector function  $\va : \D \to \DD^4_2$ we  denote by $\va^\bot$ the orthogonal projection of $\va$ on the complexified (over $\DD$)  
normal space of  $\M$. So, we have
\begin{equation} \label{PhiPr-sigma}
	\Phi^\prime=\frac{\partial\Phi}{\partial u}=\x_{uu}+\jj\x_{uv}; \qquad
  \Phi^{\prime \bot}=\x_{uu}^\bot +\jj \x_{uv}^\bot =\sigma (\x_u,\x_u)+\jj\sigma (\x_u,\x_v).
\end{equation}
The last equalities together with  \eqref{Phi-E} imply:
\begin{equation*}
{\Phi^{\prime}}^2={\Phi^{\prime \bot}}^2=
\sigma^2(\x_u,\x_u)+\sigma^2(\x_u,\x_v)+2\jj \langle \sigma(\x_u,\x_u) ,\sigma(\x_u,\x_v)\rangle .
\end{equation*}
The expressions obtained for $\Phi$ and $\Phi'$ show that the local geometry of a minimal Lorentz surface can be described in terms of these two functions.


\vspace{0.5ex}

\setcounter{equation}{0}

\section{Two approaches for introducing canonical coordinates on a minimal Lorentz surface in  $\RR^4_2$}\label{Can_Coord-MinLorSurf_R42}

We will use the null basis  $(\qq,\bar\qq)$ of $\DD$, introduced by \eqref{qq-def}, to find the relation between the function $\Phi$ and the pair  $(\alpha_1, \alpha_2)$ of null curves, defined by \eqref{MinSurf-IsoCurves}. The ''complex'' coordinate $t\in\DD$ of $\M$ is expressed with respect to the null basis as follows:
\begin{equation*}
t = u+\jj v = (u+v)\bar\qq + (u-v)\qq = t_1\bar\qq + t_2\qq.
\end{equation*}
So, we have:
\begin{equation*}
|t|^2 = t\bar t = (t_1\bar\qq + t_2\qq)(t_1\qq + t_2\bar\qq) = t_1t_2\bar\qq + t_1t_2\qq = t_1t_2.
\end{equation*}

Analogously,  $\Psi$ is expressed as:
\begin{equation*}
\begin{array}{rll}
\Psi &=&  \ds\frac{\alpha_1(u+v)+\alpha_2(u-v)}{2} + \jj \frac{\alpha_1(u+v)-\alpha_2(u-v)}{2} \\[1.5ex]
     &=&  \alpha_1(u+v)\bar\qq + \alpha_1(u-v)\qq = \alpha_1(t_1)\bar\qq + \alpha_2(t_2)\qq.
\end{array}
\end{equation*}
In the last equalities,  $(t_1,t_2)$ are the isotropic coordinates of $\M$, determined by \eqref{isoterm-isotrop}.

After differentiation we obtain the following formulas for  $\Phi$ and $\Phi'$:
\begin{equation}\label{Phi-qq}
\Phi = \alpha'_1\bar\qq + \alpha'_2\qq;  \qquad \|\Phi\|^2 = \langle\Phi , \bar\Phi \rangle = \langle\alpha'_1 , \alpha'_2\rangle.  
\end{equation}
\begin{equation}\label{Phi_p-qq}
\Phi' = \alpha''_1\bar\qq + \alpha''_2\qq;  \qquad {\Phi'}^2 = {\alpha''_1}^2\bar\qq + {\alpha''_2}^2\qq; \qquad
\big|{\Phi'}^2\big|^2 = {\alpha''_1}^2{\alpha''_2}^2.
\end{equation}

\medskip

In  \cite{Kanchev2020}, the class of minimal Lorentz surfaces of general type for which ${\Phi'}^2$ is an invertible element of $\DD$ is considered.
It follows from  \eqref{Phi_p-qq} that the last condition  is equivalent to ${\alpha''_1}^2\neq 0$ and ${\alpha''_2}^2\neq 0$.
So, we can  give the following equivalent definition of surfaces of general type. 
\begin{dfn}\label{Def-Gen_Typ-IsoCurvs}
A minimal Lorentz surface $\M$ in  $\RR^4_2$ is said to be of \textit{general type} if its corresponding null curves $\alpha_1$ and $\alpha_2$ are non-degenerate.
\end{dfn}

As we mentioned before, the property of a null curve to be non-degenerate does not depend on the parametrization. Obviously, it is invariant under motions in  $\RR^4_2$.
Consequently, the property of a minimal surface to be of general type is a geometric one: it is independent of the local parametrization and is invariant under motions in  $\RR^4_2$.

\medskip
The minimal Lorentz surfaces of general type are classified in \cite{Kanchev2020} on the base of the generalized Gauss map $\Phi$. They are  divided into three different subclasses depending on the 
 quadrant with respect to the null basis of $\DD$, where the value of ${\Phi'^\bot}^2={\Phi'}^2$ lies. 
The second formula of \eqref{Phi_p-qq} shows that the quadrant is determined by the signs of  ${\alpha''_1}^2$ and ${\alpha''_2}^2$.
Thus, we can  define the different subclasses of minimal surfaces also as follows:

\begin{dfn}\label{Def-MinSurf_kind123-IsoCurvs}
 Let $\M$ be a minimal Lorentz surface of general type in $\RR^4_2$ and  $\alpha_1$, $\alpha_2$ be its corresponding null curves. 
The surface $\M$ is said to be:
\begin{itemize}
	\item of \emph{first type}, if ${\alpha''_1}$ and ${\alpha''_2}$ are both spacelike;
	\item of \emph{second type}, if ${\alpha''_1}$ and ${\alpha''_2}$ are both timelike;
	\item  of \emph{third type}, if one of the vectors ${\alpha''_1}$ or ${\alpha''_2}$ is spacelike and the other one is timelike.
	 \end{itemize}
\end{dfn}

It follows from \eqref{t-s-IsoCurvs} that for an arbitrary null curve $\alpha$ the casual character  of $\alpha''$ does not depend on the parametrization of the curve. 
Obviously, the character  does not  change under motions in $\RR^4_2$. Hence, the classification of the minimal Lorentz surfaces, given by  Definition \ref{Def-MinSurf_kind123-IsoCurvs}, is geometric: it does not depend on the parametrization and is invariant under motions in $\RR^4_2$.

Note that the surfaces of third type can not be divided into two separate subclasses depending on  whether ${\alpha''_1}^2>0$, ${\alpha''_2}^2<0$ or ${\alpha''_1}^2<0$, ${\alpha''_2}^2>0$, since 
one of the cases is transformed  to the other one by a simultaneous re-numeration of the curves and the parameters. 
So, without loss of generality we can assume that ${\alpha''_1}^2>0$ and ${\alpha''_2}^2<0$.

In \cite{Kanchev2020}, special isothermal coordinates are introduced such that ${\Phi'^\bot}^2={\Phi'}^2$ is equal to  $1$, $-1$, or $\jj$, depending on the type of the surface. 

\begin{dfn}\label{Can_Coord-MinLorSurf_R42-Phi}
Let $\M$ be a  minimal Lorentz surface of general type in $\RR^4_2$ parametrized by  isothermal coordinates $(u,v)$. The coordinates $(u,v)$ are said to be \emph{canonical}, if the function $\Phi$ satisfies the condition: 
\begin{equation}\label{Can_Phi_R42}
{\Phi'}^2=\varepsilon,
\end{equation}
where $\varepsilon=1$ for surfaces of first type, $\varepsilon=-1$ for surfaces of second type, and 
$\varepsilon=\jj$ for surfaces of  third type.
\end{dfn}

It follows from  \eqref{qq-def} that $1 = \bar\qq + \qq$;\; $-1 = -\bar\qq - \qq$;\; $\jj = \bar\qq - \qq$.
Then, \eqref{Phi_p-qq} implies that condition  \eqref{Can_Phi_R42} is equivalent to ${\alpha''_1}^2=\pm 1$ and ${\alpha''_2}^2=\pm 1$.
The last  observations give us the idea to  introduce the concept of isotropic canonical coordinates.

\begin{dfn}\label{Can_Coord-MinLorSurf_R42-IsoCurvs}
Let $\M$ be a  minimal Lorentz surface of general type in $\RR^4_2$ with isotropic coordinates  $(t_1,t_2)$. The coordinates  $(t_1,t_2)$ are said to be \emph{canonical}, if
 $t_1$ and $t_2$ are natural parameters of the corresponding  null curves $\alpha_1$ and  $\alpha_2$, respectively.  
In the case of surfaces of third type, we assume that  ${\alpha''_1}^2=1$ and ${\alpha''_2}^2=-1$.
\end{dfn}

Thus, we can formulate the following statement.

\begin{prop}\label{Prop-CanCoord_Equiv}
Let $\M$ be a  minimal Lorentz surface of general type in $\RR^4_2$ with isothermal coordinates $(u,v)$. Then,  $(u,v)$ are canonical coordinates of $\M$  if and only if the corresponding isotropic coordinates  $(t_1,t_2)$ are canonical.
\end{prop}

Formula  \eqref{NatParam_IsoCurv} gives a natural parameter of an arbitrary non-degenerate null curve. So, we have a new proof of the existence of canonical coordinates for each minimal Lorentz surface of general type in $\RR^4_2$, and also an explicit formula for obtaining canonical coordinates.

If $t$ and  $s$ are natural parameters of one and the same null curve, then \eqref{t-s-IsoCurvs} implies that  $t'(s)=\pm 1$. Hence, the natural parameters of a null curve are related by the equality
\begin{equation*}
t=\pm s + c,
\end{equation*}
where  $c$ is a constant.

Hence, the canonical coordinates of a given minimal Lorentz surface of general type are determined uniquely up to a numeration, a sign, and an additive constant. In the case of a minimal surface of  third type, the numeration is also fixed, since  ${\alpha''_1}^2=1$ and ${\alpha''_2}^2=-1$.

\setcounter{equation}{0}

\section{Basic invariants of a minimal Lorentz surface in $\RR^4_2$ and its corresponding pair of null curves}\label{K_kappa-IsoCurvs_R42}

The basic invariants of a minimal Lorentz surface $\M$ in $\RR^4_2$ are the Gauss curvature $K$ and the curvature of the normal connection $\varkappa$.
Let $\vX_1$, $\vX_2$ be an orthonormal tangent frame field and $\n_1$, $\n_2$ -- an orthonormal normal frame field  $\M$. Without loss of generality we assume that  $\vX_1^2=\n_1^2=-\vX_2^2=-\n_2^2=\pm 1$ and for each point  $p\in\M$ the  quadruple $(\vX_1,\vX_2,\n_1,\n_2)$ is a  right oriented orthonormal frame field in $\RR^4_2$.
Then, the Gauss curvature $K$ and the curvature of the normal connection $\varkappa$ are defined by
\begin{equation*}
K = -\langle R(\vX_1,\vX_2)\vX_2, \vX_1 \rangle;  \qquad  \varkappa = \langle R^N(\vX_1,\vX_2)\n_2 , \n_1 \rangle ,
\end{equation*}
where $R$ and  $R^N$ are the curvature tensor and the normal curvature tensor, respectively. 
The Gauss curvature is expressed in terms of the second fundamental form $\sigma$  as follows:
\begin{equation*}
K = -\langle\sigma(\vX_1,\vX_1) , \sigma(\vX_2,\vX_2)\rangle + \sigma^2(\vX_1,\vX_2) = -\sigma^2(\vX_1,\vX_1) + \sigma^2(\vX_1,\vX_2).
\end{equation*}
By virtue of the Ricci equation, the curvature of the normal connection $\varkappa$  satisfies
\begin{equation*}
\varkappa = \langle [A_{\n_2},A_{\n_1}] \vX_1 , \vX_2 \rangle =
\langle A_{\n_1}\vX_1 , A_{\n_2}\vX_2 \rangle - \langle A_{\n_2}\vX_1 , A_{\n_1}\vX_2 \rangle .
\end{equation*}
Now, having in mind the last equalities and formulas \eqref{Phi-E}, \eqref{PhiPr-sigma}  we obtain the following relations between the functions $K$, $\varkappa$, and $\Phi$ (see \cite{Kanchev2020}):
\begin{equation}\label{K_kappa_Phi_R42-tl}
K = \ds\frac{-4\|\Phi\wedge\Phi'\|^2}{\|\Phi\|^6}; \quad
\varkappa = \frac{-4\,\det \big(\Phi\,,\bar\Phi\,,\Phi^\prime\,,\overline{\Phi^\prime}\,\big)}{\|\Phi\|^6}; \quad
\big|{\Phi^{\prime}}^2\big|^2 = E^4(K^2-\varkappa^2).
\end{equation}
Then,  \eqref{Phi-qq} and \eqref{Phi_p-qq} imply
\[
\|\Phi\wedge\Phi'\|^2 = \langle \alpha'_1\wedge\alpha''_1 , \alpha'_2\wedge\alpha''_2 \rangle ; \qquad  
\det \big(\Phi\,,\bar\Phi\,,\Phi^\prime\,,\overline{\Phi^\prime}\,\big) = \det \big(\alpha'_1\,, \alpha'_2\,, \alpha''_1\,, \alpha''_2\, \big).
\]
 Substituting the last expressions in \eqref{K_kappa_Phi_R42-tl}, we get that $K$ and $\varkappa$ are expressed in terms of the null curves as follows:
\begin{equation}\label{K_kappa_IsoCurvs_R42}
K         = \ds\frac{-4\,\langle\alpha'_1\wedge\alpha''_1 , \alpha'_2\wedge\alpha''_2 \rangle}{\langle\alpha'_1 , \alpha'_2\rangle^3}; \qquad
\varkappa = \ds\frac{-4\,\det (\alpha'_1\,, \alpha'_2\,, \alpha''_1\,, \alpha''_2\, )}{\langle\alpha'_1 , \alpha'_2\rangle^3}.
\end{equation}

The third equality in \eqref{K_kappa_Phi_R42-tl}, together with  \eqref{IsoCurves-cond} and \eqref{Phi_p-qq} imply:
\begin{equation*}
K^2-\varkappa^2 = \ds\frac{16\,{\alpha''_1}^2{\alpha''_2}^2}{\langle\alpha'_1 , \alpha'_2\rangle^4}.
\end{equation*}
So, we obtain the following statement.

\begin{prop}\label{DegP_kind123-K_kappa-tl} \cite{Kanchev2020}
 A minimal Lorentz surface  $\M$ in  $\RR^4_2$ is of general type if and only if $K^2-\varkappa^2\neq 0$.
Moreover,
\begin{enumerate}
	\item The surface  $\M$ is of first or second type if and only if  $K^2-\varkappa^2>0$\,. 
	\item The surface  $\M$ is of third type if and only if  $K^2-\varkappa^2<0$\,. 
\end{enumerate}
\end{prop}

The invariants  $K$ and $\varkappa$ of a minimal Lorentz surface of general type expressed in terms of the canonical isothermal coordinates give  a solution to the following system of PDEs \cite{Kanchev2020}: 
\begin{equation}\label{Nat_Eq_K_kappa_MinLorSurf_R42}
\begin{array}{lll}
\sqrt[4]{\big|K^2-\varkappa^2\big|}\; \Delta^h\ln \big|K^2-\varkappa^2\big| &=& \delta 8K\,; \\[1.5ex]
\sqrt[4]{\big|K^2-\varkappa^2\big|}\; \Delta^h\ln \left|\ds\frac{\vphantom{\mu^2}K+\varkappa}{K-\varkappa}\right|&=& \delta 4\varkappa\;;
\end{array}  \qquad\quad K^2-\varkappa^2\neq 0 \,,
\end{equation}
where $\delta=+1$ in the case $E>0$, $\delta=-1$ in the case $E<0$. 
We call \eqref{Nat_Eq_K_kappa_MinLorSurf_R42} the \textit{system of natural equations} of the minimal Lorentz surfaces in $\RR^4_2$.
In \cite{Kanchev2020}, it is shown that every solution to \eqref{Nat_Eq_K_kappa_MinLorSurf_R42} is obtained in this way.
The results obtained in \cite{Kanchev2020} can be summarized in the following theorem.

\begin{thm}\label{Thm-Nat_Eq_K_kappa_R42-tl}
Let  $\M=(\D ,\x)$ be a minimal Lorentz surface of general type in $\RR^4_2$ parametrized by canonical isothermal coordinates. 
Then, the Gauss curvature  $K$ and the curvature of the normal connection $\varkappa$ of $\M$, expressed in terms of the canonical coordinates, give a solution to the system of natural equations \eqref{Nat_Eq_K_kappa_MinLorSurf_R42}  of the minimal Lorentz surfaces in $\RR^4_2$.
If  $\hat\M$ is obtained from $\M$ by a proper motion in  $\RR^4_2$, then $\hat\M$ generates
the same solution to  system  \eqref{Nat_Eq_K_kappa_MinLorSurf_R42}. 

Conversely, let  $K$ and  $\varkappa$ be a pair of functions defined in a domain $\D\subset\RR^2$ and giving a solution to the system of natural equations \eqref{Nat_Eq_K_kappa_MinLorSurf_R42}. Then, 

(i) in the case $K^2-\varkappa^2>0$, in a neighborhood of any point $(u_0,v_0)\in\D$ there exist a unique (up to a proper
motion) minimal Lorentz  surface of first type and a unique (up to a proper motion) minimal Lorentz  surface of second type, both parametrized by canonical coordinates, for which the  functions $K$ and $\varkappa$ are the Gauss curvature
and the curvature of the normal connection, respectively. 

 (ii) in the case $K^2-\varkappa^2<0$, in a neighborhood of any point $(u_0,v_0)\in\D$ there exists a unique (up to a proper
motion) minimal Lorentz  surface of third type, parametrized by canonical coordinates, for which the functions $K$ and $\varkappa$ are the Gauss curvature and the curvature of the normal connection, respectively. 
\end{thm}

\setcounter{equation}{0}

\section{Weierstrass type representation of a null curve in $\RR^4_2$}\label{W-IsoCurvs_R42}

Let  $\alpha$ be a null curve in $\RR^4_2$. 
We will show that the nullity condition ${\alpha'}^2=0$ of the curve can be expressed in terms of three real functions in such a way that the components of $\alpha'$  to depend in a polynomial way on these functions. 
\begin{prop}\label{prop-W_IsoCurv_R42}
Let  $\alpha$ be a null curve in $\RR^4_2$ and the components of $\alpha'=(\xi_1,\xi_2,\xi_3,\xi_4)$ satisfy the condition $\xi_1-\xi_2\neq 0$. Then, $\alpha'$ can be presented in the following way: 
\begin{equation}\label{W_alphap_R42}
\alpha' = f \big(g h+1, g h-1, h - g, h + g\big),
\end{equation}
where $f\neq 0$, $g$ and $h$ are three smooth real functions determined uniquely by the curve $\alpha$ as follows:
\begin{equation}\label{fgh_alphap_R42}
f=\ds\frac{1}{2}(\xi_1-\xi_2); \qquad
g=\ds\frac{\xi_4-\xi_3}{\xi_1-\xi_2}; \qquad
h=\ds\frac{\xi_4+\xi_3}{\xi_1-\xi_2}.
\end{equation}

Conversely, if  $(f,g,h)$ are three smooth real functions such that $f\neq 0$, then there exists a null curve $\alpha$ in  $\RR^4_2$ such that 
$\alpha'$ is presented by the given functions according to  \eqref{W_alphap_R42}  and the condition $\xi_1-\xi_2\neq 0$ is satisfied. 
\end{prop}

\begin{proof}
 Let $(f,g,h)$ be three smooth real functions defined in an interval of  $\RR$ and  $f\neq 0$.
We consider the curve $\alpha$ defined by \eqref{W_alphap_R42}.
To simplify the calculations, we  introduce the following  vector function:
\begin{equation}\label{Def-va_IsoCurv_R42}
\va=\frac{\alpha'}{f}=\big(g h+1, g h-1, h - g, h + g \big).
\end{equation}
Then, 
\begin{equation}\label{alpha-va_IsoCurv_R42}
\alpha'=f\va; \qquad \alpha'^2=f^2\va^2; \qquad \alpha''=f'\va + f\va'. 
\end{equation}
Direct computations show  that $\va^2=0$. Moreover, \eqref{alpha-va_IsoCurv_R42} implies ${\alpha'}^2=0$ and \eqref{Def-va_IsoCurv_R42} implies $\xi_1-\xi_2=2f$. Hence,   $\alpha'\neq 0$ since $f\neq 0$.
So, in the case $f\neq 0$, formula  \eqref{W_alphap_R42} defines a null curve in $\RR^4_2$.
It follows directly from \eqref{W_alphap_R42} that the functions $f$, $g$, and  $h$ are expressed in terms of the  components of  $\alpha'$ by formulas \eqref{fgh_alphap_R42}. 

\smallskip

Conversely, let $\alpha$ be a null curve such that  $\xi_1-\xi_2\neq 0$.  We consider the real functions $(f,g,h)$ defined by 
\eqref{fgh_alphap_R42}. Then, by direct computations we see that  $\alpha'$ is expressed as given in \eqref{W_alphap_R42}.
\end{proof}

Formula  \eqref{W_alphap_R42} is analogous to the classical  Weierstrass representation of minimal surfaces in $\RR^3$. For this reason, we will call it the \textit{Weierstrass representation} of null curves in $\RR^4_2$. We will briefly say that the corresponding curve is \textit{generated by the triple of functions $(f,g,h)$}.

\begin{remark}\label{rem-W_IsoCurv_R42}
The condition $\xi_1-\xi_2\neq 0$ on $\alpha$ in the above proposition is not an essential one. 
If we assume that $\xi_1-\xi_2 = 0$ for a given null curve, then  by a proper motion it can be transformed to a curve satisfying the condition $\xi_1-\xi_2\neq 0$. Since all geometric properties and formulas which will be considered below are invariant under proper motions in $\RR^4_2$, they will also be valid for curves with $\xi_1-\xi_2 = 0$. Therefore, further we will not explicitly write this condition.
\end{remark}

Now, we will obtain the transformation formulas for the functions $(f,g,h)$ participating in the Weierstrass
representation formula \eqref{W_alphap_R42}. We will show that the functions  $g$ and $h$ are transformed by  linear-fractional transformations with real coefficients. For this purpose we will use some basic maps and formulas from the spinor theory in $\RR^4_2$. 

 To each vector  $\x$ in $\RR^4_2$ we assign   $2\times 2$ real matrix $S$ as follows: 
\begin{equation}\label{Spin_S-x_R42}
\x=(x_1,x_2,x_3,x_4)
 \ \leftrightarrow \ 
S=\left(
\begin{array}{cc}
   x_4 - x_3  & x_1 + x_2\\
   x_1 - x_2  & x_4 + x_3
\end{array}
\right) .
\end{equation}
It  can easily be seen that the correspondence given above is a linear isomorphism between  $\RR^4_2$ and the space of all
$2\times 2$ real matrices. In addition, direct computations show that  $\det S = \x^2$. 
This means that any linear operator acting
in the space of $2\times 2$ real matrices and preserving the determinant gives us an orthogonal operator in $\RR^4_2$. 

 If $(B_1,B_2)$ is a pair of $2\times 2$ matrices from the group $\mathbf{SL}(2,\RR)$, then the equality $\det B_1 S B_2^{-1} = \det S$ holds true. Hence, each such pair corresponds to an orthogonal matrix  $A$ from $\mathbf{O}(2,2,\RR)$.  
Therefore, we have a group homomorphism  $(B_1,B_2) \rightarrow A$, which is determined as follows: 
\begin{equation}\label{Spin_B1B2-A_R42}
\hat S = B_1 S B_2^{-1} \ \rightarrow \ \hat\x = A\x\,.
\end{equation}

The homomorphism between  $\mathbf{SL}(2,\RR) \times \mathbf{SL}(2,\RR)$ and $\mathbf{O}(2,2,\RR)$  is briefly called a spinor map. In the spinor theory it is proved that the kernel of the spinor map consists of two elements: $(I,I)$ and $(-I,-I)$,
where  $I$ denotes the unit matrix, and the image of the spinor map coincides with the connected component of the unit element (the identity) of  $\mathbf{O}(2,2,\RR)$ \cite{TdCastillo-1}. 
This is the group of the proper  orthochronous motions in $\RR^4_2$, denoted by $\mathbf{SO}^{+}(2,2,\RR)$. 
It follows from the type of the kernel and the image of the spinor map \eqref{Spin_B1B2-A_R42} 
 that it induces the following group isomorphism:
\begin{equation*}
\mathbf{SL}(2,\RR) \times \mathbf{SL}(2,\RR)/\{(I,I),(-I,-I)\}\ \cong \ \mathbf{SO}^{+}(2,2,\RR)\,.
\end{equation*}

Now, we will go back to the null curves and prove the following statement.  

\begin{thm}\label{W-proper_move_R42} 
Let $\hat\alpha$ and $\alpha$ be two null curves in  $\RR^4_2$ given by Weierstrass representation of the form \eqref{W_alphap_R42}.  
Then, the following conditions are equivalent:
\begin{enumerate}
	\item $\hat\alpha$ and $\alpha$ are related by a proper motion in $\RR^4_2$ of the following form:\\
	$\hat\alpha(t)=A\alpha(t)+\vb$, where $A \in \mathbf{SO}(2,2,\RR)$ and $\vb \in \RR^4_2$.
	\item The functions in the Weierstrass representation formulas of $\hat\alpha$ and $\alpha$ are related as follows: 
\begin{gather}\label{hatfgh_fgh-proper_move_R42}
\begin{array}{l}
\hat f = f(c_1 g + d_1)(c_2 h + d_2); \\[0.7ex]
\hat g = \ds\frac{a_1 g + b_1}{c_1 g + d_1}; \quad 
\hat h = \ds\frac{a_2 h + b_2}{c_2 h + d_2};
\end{array}
\qquad 
\begin{array}{l}
a_1,b_1,c_1,d_1,a_2,b_2,c_2,d_2 \in \RR; \\
a_1d_1-b_1c_1=a_2d_2-b_2c_2=\pm 1.
\end{array}
\end{gather}
\end{enumerate}
\end{thm}

\begin{proof}
Let the curve $\hat\alpha$ be obtained by $\alpha$ trough a proper  orthochronous motion in $\RR^4_2$ of  the form
 $\hat\alpha(t)=A\alpha(t)+\vb$, where $A \in \mathbf{SO}^{+}(2,2,\RR)$ and $\vb \in \RR^4_2$, i.e.  $\hat\alpha'=A\alpha'$.
Let $\alpha'=(\xi_1,\xi_2,\xi_3,\xi_4)$ and  consider the matrix $S_{\alpha}$, induced by  $\alpha'$ according to \eqref{Spin_S-x_R42}:
\begin{equation*}
S_{\alpha}=\left(
\begin{array}{cc}
   \xi_4 - \xi_3  &  \xi_1 + \xi_2\\
   \xi_1 - \xi_2  &  \xi_4 + \xi_3 
\end{array}
\right) .
\end{equation*}
Denote by  $(B_1,B_2)$ either of the two pairs of matrices in $\mathbf{SL}(2,\RR) \times \mathbf{SL}(2,\RR)$, which correspond to $A$ by the homomorphism \eqref{Spin_B1B2-A_R42}. If $S_{\hat\alpha}$ is the matrix induced by $\hat\alpha$, then according to \eqref{Spin_B1B2-A_R42} we have:
\begin{equation}\label{hatSalpha_Salpha_R42}
S_{\hat\alpha} = B_1 S_{\alpha} B_2^{-1}.
\end{equation}

\smallskip
Conversely,  if an equality of the form \eqref{hatSalpha_Salpha_R42} is valid for two null curves $\hat\alpha$ and $\alpha$, then  according to \eqref{Spin_B1B2-A_R42} we have $\hat\alpha'=A\alpha'$, which is equivalent to $\hat\alpha(t)=A\alpha(t)+\vb$, where $A \in \mathbf{SO}^{+}(2,2,\RR)$ and $\vb \in \RR^4_2$.

\smallskip
Now, suppose that $\alpha$ is given by Weierstrass  representation of the form \eqref{W_alphap_R42}. Then, by direct computation we obtain
\begin{equation*}
S_{\alpha}=\left(
\begin{array}{ll}
\vspace{2mm}
   2fg      &  2fgh \\
	\vspace{2mm}
   2f       &  2fh
\end{array}
\right).
\end{equation*}
Denote by $s_{ij}$, $i, j = 1, 2$ the elements of $S_{\alpha}$. Then, for the corresponding functions $f$, $g$, and $h$ we have: 
\begin{equation}\label{fgh_s1234_R42}
f=\frac{1}{2}s_{21}; \qquad g=\frac{s_{11}}{s_{21}}; \qquad h=\frac{s_{22}}{s_{21}}.
\end{equation}

We already know that $S_{\alpha}$ is transformed in accordance with \eqref{hatSalpha_Salpha_R42} under a proper orthochronous motion. If we denote the elements of the matrices $B_1$ and $B_2$ as follows:
\begin{equation*}
B_1=\left(
\begin{array}{rr}
     a_1  &  b_1 \\
     c_1  &  d_1
\end{array}
\right); \quad
B_2=\left(
\begin{array}{rr}
     a_2  &  -b_2 \\
   - c_2  &   d_2
\end{array}
\right);
\qquad 
\begin{array}{l}
a_1,b_1,c_1,d_1,a_2,b_2,c_2,d_2 \in \RR; \\
a_1d_1-b_1c_1=a_2d_2-b_2c_2=1,
\end{array}
\end{equation*}
then, by use of \eqref{hatSalpha_Salpha_R42} we get:
\begin{equation*}
S_{\hat\alpha}=\left(
\begin{array}{rr}
\vspace{2mm}
     2f(a_1 g + b_1)(c_2 h + d_2)   &   2f(a_1 g + b_1)(a_2 h + b_2) \\
		\vspace{2mm}
     2f(c_1 g + d_1)(c_2 h + d_2)   &   2f(c_1 g + d_1)(a_2 h + b_2)
\end{array}
\right).
\end{equation*}
Now, applying  \eqref{fgh_s1234_R42} to $\hat f$, $\hat g$, and  $\hat h$, we obtain the transformation formulas of the functions in the Weierstrass representation \eqref{W_alphap_R42}  under a proper orthochronous motion of $\alpha$ in $\RR^4_2$: 
\begin{gather}\label{hatfgh_fgh-orthchr_move_R42}
\begin{array}{l}
\hat f = f(c_1 g + d_1)(c_2 h + d_2) \,; \\[0.7ex]
\hat g = \ds\frac{a_1 g + b_1}{c_1 g + d_1}\,; \quad 
\hat h = \ds\frac{a_2 h + b_2}{c_2 h + d_2}\,;
\end{array}
\qquad 
\begin{array}{l}
a_1,b_1,c_1,d_1,a_2,b_2,c_2,d_2 \in \RR\,; \\
a_1d_1-b_1c_1=a_2d_2-b_2c_2=1 \,.
\end{array}
\end{gather}

\smallskip
It is easy to see that the opposite is also true: if the functions in the Weierstrass representation \eqref{W_alphap_R42} 
of two null curves satisfy formulas \eqref{hatfgh_fgh-orthchr_move_R42}, then \eqref{hatSalpha_Salpha_R42}
 is fulfilled and therefore, the curves are related by a proper orthochronous motion in $\RR^4_2$.

\smallskip
Now, let us consider the case of a proper non-orthochronous motion in $\RR^4_2$. Such a case occurs, if we change the signs of the third and the fourth coordinate. Let $\hat\alpha$ be obtained by $\alpha$ under such transformation. Then, \eqref{W_alphap_R42}  implies that
$\hat f = f$, $\hat g = -g$,  $\hat h = -h$. So, the functions $f$, $g$, $h$ change analogously to  \eqref{hatfgh_fgh-orthchr_move_R42}, the difference being that the linear-fractional transformations are given by matrices with  determinant $-1$.
 Each proper non-orthochronous motion can be obtained as a composition of this special motion and a proper orthochronous motion in $\RR^4_2$. 
Consequently, if two null curves are related by a proper non-orthochronous motion, then the corresponding functions in the Weierstrass representation are changed in accordance with formulas analogous to  \eqref{hatfgh_fgh-orthchr_move_R42}, where the matrices of the linear-fractional transformations of $g$ and  $h$ have determinants  $-1$.

Summarizing the results for proper orthochronous and non-orthochronous motions, we finish the proof of the theorem.
\end{proof}

The formulas obtained so far are valid for an arbitrary parametrization of the null curve. Now, we shall consider a non-degenerate null curve $\alpha$ parametrized by a natural parameter, i.e. ${\alpha''}^2=\pm 1$.

\begin{prop}\label{IsoCurv-nondeg_fgh_R42}
Let $\alpha$ be a null curve in $\RR^4_2$ with Weierstrass representation \eqref{W_alphap_R42}. Then, $\alpha$ is non-degenerate if and only if $g'h' \neq 0$ at each point. 
\end{prop}

\begin{proof}
Let  $\va$ be the vector defined by \eqref{Def-va_IsoCurv_R42}. Since $\va^2=0$, we have  $\langle \va , \va' \rangle = 0$.
Then \eqref{alpha-va_IsoCurv_R42} implies ${\alpha''}^2=f^2{\va'}^2$. By direct computation it follows that  ${\va'}^2=4g'h'$.
Hence,
\begin{equation}\label{alphapp^2_fgh_R42}
{\alpha''}^2=4f^2g'h'.
\end{equation}
Having in mind that  $f\neq 0$, from \eqref{alphapp^2_fgh_R42} we obtain that  ${\alpha''}^2\neq 0$ if and only if $g'h' \neq 0$\,.
\end{proof}

In the case of a natural parameter, only two functions remain in the Weierstrass representation of the null curve. In this case, the following statement holds true.

\begin{prop}\label{prop-W_natparm_IsoCurv_R42}
Let $\alpha$ be   a null curve in  $\RR^4_2$ parametrized by a natural parameter, i.e. ${\alpha''}^2=\pm 1$. 
Then, $\alpha$ has the following Weierstrass representation:
\begin{equation}\label{W_natparm_alphap_R42}
\alpha'= \ds\frac{\omega}{2\sqrt{|g'h'|}}\left(gh+1, gh-1, h-g, h+g\right),
\end{equation}
where $g$ and $h$ are smooth real functions satisfying $g'h' \neq 0$, and $\omega=\pm 1$.
The functions $g$ and  $h$ as well as $\omega$ are determined uniquely by $\alpha$ in accordance with \eqref{fgh_alphap_R42}.

 Conversely, if $(g,h)$ is a pair of smooth real functions satisfying $g'h' \neq 0$ and $\omega=\pm 1$, then there exists a null curve  $\alpha$ in $\RR^4_2$ parametrized by a natural parameter, such that $\alpha'$ is expressed by the given functions in the form \eqref{W_natparm_alphap_R42}. 
\end{prop}

\begin{proof}
Let  $\alpha$ be a non-degenerate null curve parametrized by a natural parameter. It follows from \eqref{W_alphap_R42} and \eqref{alphapp^2_fgh_R42} that $4f^2g'h'=\pm 1$, which is equivalent to   
 $4f^2|g'h'|=1$. Hence, we obtain
\begin{equation}\label{fgh_natparm_R42}
f = \ds\frac{\omega}{2\sqrt{|g'h'|}}; \qquad \omega=\pm 1,
\end{equation}
where the sign of $\omega$ is the same as the sign of $f$.
Using \eqref{fgh_natparm_R42} and  \eqref{W_alphap_R42} we get \eqref{W_natparm_alphap_R42}.

 The opposite is easy to be seen: if $\alpha'$ is of the form \eqref{W_natparm_alphap_R42}, then \eqref{fgh_natparm_R42} holds true, which implies ${\alpha''}^2=4f^2g'h'=\pm 1$.  
Hence, $\alpha$ is parametrized by a natural parameter. 
\end{proof}

\begin{remark}\label{rem2-W_natparm_IsoCurv_R42}
For a fixed pair of functions $(g,h)$ and a different choice of the sign of $\omega$ formula \eqref{W_natparm_alphap_R42} represents two null curves  differing in the sign of  $\alpha'$.
Hence, the curves thus obtained are related by a proper orthochronous motion in $\RR^4_2$.
\end{remark}

\begin{example}\label{exmp-IsoCurvs_R42_Enp} Now, we will present examples of curves with representation of the form \eqref{W_natparm_alphap_R42}.
Let us consider the linear functions $g(t)=kt$;\; $h(t)=lt$, where $k$ and $l$ are non-zero real constants. 
Applying \eqref{W_natparm_alphap_R42} we obtain a family of functions $\alpha_{k,l}$ (parametrized by a natural parameter $t$), defined by
\begin{equation}\label{exmp-W_natparm_alphap_R42}
\alpha'_{k,l}(t)= \frac{\omega}{2\sqrt{|kl|}} \left(klt^2+1, klt^2-1, (l-k)t, (l+k)t\right).
\end{equation} 
Note that according to Theorem \ref{W-proper_move_R42} each of the curves $\alpha_{k,l}$ is congruent to either the curve  $\alpha_{1,1}$ or  $\alpha_{1,-1}$.
\end{example}

\setcounter{equation}{0}

\section{Weierstrass type representation of  minimal Lorentz surfaces in $\RR^4_2$}\label{W-MinLorSurf_R42}

It is known that the minimal Lorentz surfaces in $\RR^3_1$ have a Weierstrass type representation analogous to the classical one in $\RR^3$ but in terms of functions  holomorphic over the algebra $\DD$    (see \cite{Kond}). Such a representation has a natural analogue also  in  $\RR^4_2$. In the present section,  using the obtained Weierstrass type representation for null  curves, we will find a Weierstrass type representation for minimal Lorentz surfaces in  $\RR^4_2$ by applying the approach given in \cite{K-M-1}.  
This approach has the obvious advantage that the considered surface and all its invariants are expressed in terms of real-valued functions.

\begin{thm}\label{prop-W_MinLorSurf_R42}
Let $\M$ be a  minimal Lorentz surface in  $\RR^4_2$ and  $(\alpha_1, \alpha_2)$  be its corresponding pair of null curves. 
Then,
\begin{equation}\label{W_MinLorSurf_R42}
\alpha'_i = f_i \big(g_i h_i + 1,  g_i h_i - 1,  h_i - g_i, h_i + g_i\big); \qquad i=1, 2,
\end{equation}
where $(f_i,g_i,h_i)$; $i=1, 2$ are  two triples of smooth real functions such that
\begin{equation}\label{W_cond_MinLorSurf_R42}
f_1(t_1)\neq 0; \qquad f_2(t_2)\neq 0; \qquad g_1(t_1) \neq g_2(t_2); \qquad h_1(t_1) \neq h_2(t_2).
\end{equation}

Conversely, if $(f_i,g_i,h_i)$; $i=1, 2$ are  two triples of  smooth real functions satisfying conditions \eqref{W_cond_MinLorSurf_R42}, then, there exists a minimal Lorentz surface in $\RR^4_2$ such that its corresponding null curves have a Weierstrass type representation of the form  
 \eqref{W_MinLorSurf_R42} expressed by the given functions. 
\end{thm}

\begin{proof}
Let $\M$ be a minimal Lorentz surface in $\RR^4_2$. 
According to Remark \ref{rem-W_IsoCurv_R42} we may assume that each of its corresponding null  curves $\alpha_1$ and $\alpha_2$ satisfies the condition  $\xi_1-\xi_2\neq 0$.
Then, from Proposition \ref{prop-W_IsoCurv_R42} it follows directly that \eqref{W_MinLorSurf_R42} is fulfilled.
Let $\va_1$ and $\va_2$ be the vector functions defined by \eqref{Def-va_IsoCurv_R42} and corresponding to $\alpha_1$ and $\alpha_2$, respectively.  
It follows from  \eqref{IsoCurves-cond} and \eqref{alpha-va_IsoCurv_R42} that 
$E = \frac{1}{2}\langle \alpha'_1 , \alpha'_2 \rangle = \frac{1}{2}f_1f_2 \langle \va_1 , \va_2 \rangle$.
By direct computations it can be found  that $\langle\va_1 ,\va_2\rangle = - 2(g_1-g_2)(h_1-h_2)$, which implies
\begin{equation}\label{E-fgh_MinLorSurf_R42}
E= \frac{1}{2} \langle\alpha'_1 , \alpha'_2\rangle = - f_1f_2(g_1-g_2)(h_1-h_2).
\end{equation}
The inequalities \eqref{W_cond_MinLorSurf_R42} are valid, since $E\neq 0$.

Conversely, if $(f_i,g_i,h_i)$; $i=1, 2$ are  two triples satisfying conditions \eqref{W_cond_MinLorSurf_R42}, then 
 \eqref{W_MinLorSurf_R42} defines two null curves  $\alpha_1$ and $\alpha_2$ such that $\alpha'_1 \alpha'_2\neq 0$ in view of 
 \eqref{W_cond_MinLorSurf_R42} and \eqref{E-fgh_MinLorSurf_R42}. 
Consequently, formula  \eqref{MinSurf-IsoCurves} gives us  the desired minimal Lorentz surface.
\end{proof}

We will call  formula \eqref{W_MinLorSurf_R42} the \textit{Weierstrass representation} of a minimal Lorentz surface in  $\RR^4_2$.

\medskip

Now, we will express the Gauss curvature $K$ and the curvature of the normal connection  $\varkappa$ in terms of the functions in the Weierstrass representation.
Using \eqref{K_kappa_IsoCurvs_R42} and \eqref{alpha-va_IsoCurv_R42} we get:
\[
\langle\alpha'_1\wedge\alpha''_1 , \alpha'_2\wedge\alpha''_2\rangle=
(f_1f_2)^2 \langle\va_1\wedge\va'_1 , \va_2\wedge\va'_2\rangle=
(f_1f_2)^2\big(\langle\va_1,\va_2\rangle\langle\va'_1,\va'_2\rangle - \langle\va_1,\va'_2\rangle\langle\va'_1,\va_2\rangle\big);
\]
\[
\det \big(\alpha'_1\,, \alpha'_2\,, \alpha''_1\,, \alpha''_2\, \big)=
(f_1f_2)^2 \det \big(\va_1\,, \va_2\,, \va'_1\,, \va'_2\, \big);
\]
\[
\langle\alpha'_1,\alpha'_2\rangle^3=(f_1f_2)^3 \langle\va_1,\va_2\rangle^3.
\]
Substituting the last expressions in \eqref{K_kappa_IsoCurvs_R42} we obtain:
\begin{equation*}
K = \ds\frac{-4\big(\langle\va_1,\va_2\rangle\langle\va'_1,\va'_2\rangle - \langle\va_1,\va'_2\rangle\langle\va'_1,\va_2\rangle\big)}
            {f_1f_2 \langle\va_1,\va_2\rangle^3}; \qquad
\varkappa = \frac{-4\,\det \big(\va_1\,, \va_2\,, \va'_1\,, \va'_2\, \big)}{f_1f_2 \langle\va_1,\va_2\rangle^3}.
\end{equation*}
Now, by direct computations it follows that 
\[
\langle\va_1,\va_2\rangle\langle\va'_1,\va'_2\rangle - \langle\va_1,\va'_2\rangle\langle\va'_1,\va_2\rangle=4\big(g'_1g'_2(h_1-h_2)^2+h'_1h'_2(g_1-g_2)^2\big);
\]
\[
\det \big(\va_1\,, \va_2\,, \va'_1\,, \va'_2\, \big)=4\big(g'_1g'_2(h_1-h_2)^2-h'_1h'_2(g_1-g_2)^2\big);
\]
\[
\langle\va_1,\va_2\rangle = -2(g_1-g_2)(h_1-h_2).
\]
Using the last equalities  we obtain that $K$ and  $\varkappa$ are expressed as follows: 
\begin{equation}\label{K_kappa-fgh_MinLorSurf_R42}
\begin{array}{llr}
K         &=& \ds\frac{2}{f_1f_2 (g_1-g_2)(h_1-h_2)}
              \left(\ds\frac{g'_1g'_2}{(g_1-g_2)^2}+\ds\frac{h'_1h'_2}{(h_1-h_2)^2}\right);\\[3ex]
\varkappa &=& \ds\frac{2}{f_1f_2 (g_1-g_2)(h_1-h_2)}
              \left(\ds\frac{g'_1g'_2}{(g_1-g_2)^2}-\ds\frac{h'_1h'_2}{(h_1-h_2)^2}\right).
\end{array}
\end{equation}

\medskip

 The considerations above are valid for any minimal Lorentz surface in $\RR^4_2$. Now, we assume that 
 $\M$ is of general type according to Definition \ref{Def-Gen_Typ-IsoCurvs} and is parametrized by canonical isotropic coordinates. Then, by Definition \ref{Can_Coord-MinLorSurf_R42-IsoCurvs} we have that its corresponding null curves are parametrized by natural parameters. These curves have a Weierstrass representation of the form  \eqref{W_natparm_alphap_R42}. Then, analogously to the case of arbitrary coordinates, from Proposition \ref{prop-W_natparm_IsoCurv_R42} we obtain

\begin{prop}\label{prop-CanW_MinLorSurf_R42}
Let $\M$ be a minimal Lorentz surface of general type in $\RR^4_2$ parametrized by canonical isotropic coordinates. Then, $\M$ has the following Weierstrass type representation:   
\begin{equation}\label{CanW_MinLorSurf_R42}
\alpha'_i= \ds\frac{\omega_i}{2 \sqrt{|g'_ih'_i|}} \left(g_ih_i+1, g_ih_i-1, h_i-g_i, h_i+g_i\right); \qquad \omega_i=\pm 1,
\end{equation}
where $(g_i,h_i)$; $i=1, 2$ are two pairs of smooth real functions such that: 
\begin{equation}\label{CanW_cond_MinLorSurf_R42}
g'_1(t_1)h'_1(t_1)\neq 0; \qquad g'_2(t_2)h'_2(t_2)\neq 0; \qquad g_1(t_1) \neq g_2(t_2); \qquad h_1(t_1) \neq h_2(t_2),
\end{equation}
and, in addition, if $\M$  is of third type  according to  Definition \ref{Def-MinSurf_kind123-IsoCurvs}, then $g'_1h'_1>0$ and $g'_2h'_2<0$.

Conversely, if $(g_i,h_i)$; $i=1, 2$ are  two pairs of smooth real functions satisfying \eqref{CanW_cond_MinLorSurf_R42}, then there exists a minimal Lorentz surface of general type in $\RR^4_2$ parametrized by  canonical isotropic coordinates and having  Weierstrass representation
\eqref{CanW_MinLorSurf_R42} expressed by the given  functions.  
\end{prop}

We will briefly call formula  \eqref{CanW_MinLorSurf_R42} the \emph{canonical Weierstrass representation} of a minimal Lorentz surface of general type in  $\RR^4_2$.

\medskip

 The curves  $\alpha_1$ and  $\alpha_2$ satisfy  equalities of the form \eqref{fgh_natparm_R42}. Applying them to formula \eqref{E-fgh_MinLorSurf_R42}, we get:
\begin{equation*}
E=-\frac{\omega_1\omega_2 (g_1-g_2)(h_1-h_2)}{4\sqrt{|g'_1h'_1g'_2h'_2|}}.
\end{equation*}
Since $|-\omega_1\omega_2|=1$, the last formula implies 
\begin{equation}\label{omega-delta_MinLorSurf_R42}
-\omega_1\omega_2 (g_1-g_2)(h_1-h_2) = \delta |(g_1-g_2)(h_1-h_2)|, 
\end{equation}
where $\delta=+1$ in the case $E>0$, and $\delta=-1$ in the case $E<0$.
Hence, in the case of canonical coordinates, $E$ is expressed as follows:
\begin{equation}\label{E-gh_MinLorSurf_R42}
E=\frac{\delta |(g_1-g_2)(h_1-h_2)|}{4\sqrt{|g'_1h'_1g'_2h'_2|}}.
\end{equation}

Applying \eqref{fgh_natparm_R42} and \eqref{omega-delta_MinLorSurf_R42}  to formulas \eqref{K_kappa-fgh_MinLorSurf_R42}  we obtain that the curvatures  $K$ and $\varkappa$ are expressed in canonical coordinates as follows:
\begin{equation}\label{K_kappa-gh_MinLorSurf_R42}
\begin{array}{llr}
K         &=& \ds\frac{-\delta 8\sqrt{|g'_1h'_1g'_2h'_2|}}{|(g_1-g_2)(h_1-h_2)|}
              \left(\ds\frac{g'_1g'_2}{(g_1-g_2)^2}+\ds\frac{h'_1h'_2}{(h_1-h_2)^2}\right);\\[3ex]
\varkappa &=& \ds\frac{-\delta 8\sqrt{|g'_1h'_1g'_2h'_2|}}{|(g_1-g_2)(h_1-h_2)|}
              \left(\ds\frac{g'_1g'_2}{(g_1-g_2)^2}-\ds\frac{h'_1h'_2}{(h_1-h_2)^2}\right).
\end{array}
\end{equation}

\smallskip
Finally, we will give examples of surfaces with representation of  type \eqref{CanW_MinLorSurf_R42}.
We will use the curves $\alpha_{k,l}$ defined by \eqref{exmp-W_natparm_alphap_R42}. 
Since the obtained surfaces are generated by linear functions, they are the natural analogue of the classical Enneper surface  in $\RR^3$.

\vskip 2mm
\begin{example}\label{exmp-MinLorSurf_R42_Enp1}
First, let $\M_1$ be the surface corresponding to the pair  $(\alpha_{2,1}\,, \alpha_{1,2})$,
where $(\omega_1\,,\omega_2)$ is chosen in such a way that $E<0$. $\M_1$ is a minimal surface of first type, since 
 $g'_1h'_1=2>0$ and $g'_2h'_2=2>0$. Applying \eqref{K_kappa-gh_MinLorSurf_R42}  we obtain the curvatures 
$K$ and $\varkappa$ of $\M_1$:
\begin{equation}\label{exmp-K_kappa-MinLorSurf_kind1_R42}
K         = \ds\frac{32(t_1-2t_2)^2+32(2t_1-t_2)^2}{|(t_1-2t_2)(2t_1-t_2)|^3}; \qquad
\varkappa = \ds\frac{32(t_1-2t_2)^2-32(2t_1-t_2)^2}{|(t_1-2t_2)(2t_1-t_2)|^3}.
\end{equation}
\end{example}

\vskip 2mm
\begin{example}\label{exmp-MinLorSurf_R42_Enp2}
Now, let $\M_2$ be the surface corresponding to the pair $(\alpha_{2,-1}\,, \alpha_{1,-2})$,
where $(\omega_1\,,\omega_2)$ is chosen in such a way that $E<0$.
In this case, $\M_2$ is a minimal surface of second type, since  $g'_1h'_1=-2<0$ and $g'_2h'_2=-2<0$. 
The curvatures  $K$ and $\varkappa$ of $\M_2$ are expressed again by  \eqref{exmp-K_kappa-MinLorSurf_kind1_R42}, since formulas \eqref{K_kappa-gh_MinLorSurf_R42} do not change if we replace $(g_i,h_i)$ with $(g_i,-h_i)$.
\end{example}

\vskip 2mm
\begin{example}\label{exmp-MinLorSurf_R42_Enp3}
Finally, let $\M_3$ be the surface  corresponding to the pair $(\alpha_{2,1}\,, \alpha_{1,-2})$,
where again $(\omega_1\,,\omega_2)$ is chosen in such a way that $E<0$. In this case, $\M_3$  is of  third type, since 
 $g'_1h'_1=2>0$ and $g'_2h'_2=-2<0$. The curvatures  $K$ and $\varkappa$ of $\M_3$ are given by:
\begin{equation*}\label{exmp-K_kappa-MinLorSurf_kind3_R42}
K         = \ds\frac{32(t_1+2t_2)^2-32(2t_1-t_2)^2}{|(t_1+2t_2)(2t_1-t_2)|^3}; \qquad
\varkappa = \ds\frac{32(t_1+2t_2)^2+32(2t_1-t_2)^2}{|(t_1+2t_2)(2t_1-t_2)|^3}.
\end{equation*}
\end{example}

\setcounter{equation}{0}

\section{Explicit solving  of the system of natural equations of minimal Lorentz surfaces in $\RR^4_2$}\label{sect-Sol_NatEq_MinLorSurf_R42}

Now, we will show that formulas  \eqref{K_kappa-gh_MinLorSurf_R42} for the curvatures $K$ and $\varkappa$ can be interpreted as formulas
for the general solution of the system of PDEs \eqref{Nat_Eq_K_kappa_MinLorSurf_R42}. 

\begin{thm}\label{Sol_NatEq_MinLorSurf_R42}
Let $(K,\varkappa)$ be a pair of functions which is a solution to the system of natural equations \eqref{Nat_Eq_K_kappa_MinLorSurf_R42} of  minimal Lorentz surfaces in $\RR^4_2$. 
Then, at least locally, $K$ and $\varkappa$ are expressed by formulas \eqref{K_kappa-gh_MinLorSurf_R42},
where $(g_i,h_i)$; $i=1, 2$ is a quadruple of real functions of one variable satisfying  conditions  \eqref{CanW_cond_MinLorSurf_R42},
 $g_1$ and $h_1$ depend on $u+v$, $g_2$ and $h_2$ depend on  $u-v$. 

Conversely, each such quadruple of real functions of one variable generates according to formulas \eqref{K_kappa-gh_MinLorSurf_R42} a solution  $(K,\varkappa)$ to the system of natural equations \eqref{Nat_Eq_K_kappa_MinLorSurf_R42} of  
minimal Lorentz surfaces in $\RR^4_2$.
\end{thm}

\begin{proof} 
Let $(K,\varkappa)$ be a solution to the system of natural equations  \eqref{Nat_Eq_K_kappa_MinLorSurf_R42}.
Using Theorem \ref{Thm-Nat_Eq_K_kappa_R42-tl}, we get that locally there exists a minimal Lorentz surface of general type in $\RR^4_2$  parametrized by canonical coordinates and 
having  the given functions $K$ and $\varkappa$ as the Gauss and the normal curvature, respectively. Applying Proposition 
 \ref{prop-CanW_MinLorSurf_R42}  we obtain that this surface has a canonical Weierstrass representation of the form  \eqref{CanW_MinLorSurf_R42}. 
Thus, we have the quadruple of real functions  $(g_i,h_i)$; $i=1, 2$ satisfying  conditions  \eqref{CanW_cond_MinLorSurf_R42} and formulas \eqref{K_kappa-gh_MinLorSurf_R42} are valid for these functions.

Now, let $(g_i,h_i)$; $i=1, 2$ be a quadruple of real functions satisfying  \eqref{CanW_cond_MinLorSurf_R42}. Then, 
 \eqref{CanW_MinLorSurf_R42} determines a minimal Lorentz surface of general type $\M$ in $\RR^4_2$.
We may assume that $\omega_1$ and $\omega_2$ are chosen in such a way that the sign of $E$ coincides with the sign of 
$\delta$ in system \eqref{Nat_Eq_K_kappa_MinLorSurf_R42}.
If  $g'_1h'_1g'_2h'_2>0$, then  \eqref{alphapp^2_fgh_R42} implies that $\M$ is of  first or second type and the  coordinates are canonical. 
The Gauss curvature and the normal curvature of $\M$ expressed by \eqref{K_kappa-gh_MinLorSurf_R42} give a solution to system 
 \eqref{Nat_Eq_K_kappa_MinLorSurf_R42} according to Theorem \ref{Thm-Nat_Eq_K_kappa_R42-tl}.
If $g'_1h'_1>0$ and $g'_2h'_2<0$, then $\M$ is of  third type and is parametrized by canonical coordinates. Again, according to Theorem \ref{Thm-Nat_Eq_K_kappa_R42-tl} formulas 
\eqref{K_kappa-gh_MinLorSurf_R42} give a solution to system \eqref{Nat_Eq_K_kappa_MinLorSurf_R42}.
If $g'_1h'_1<0$ and $g'_2h'_2>0$, then the quadruple $(g_i,-h_i)$; $i=1, 2$ fulfills the conditions of the previous case
and therefore determines  a solution to the system by formulas \eqref{K_kappa-gh_MinLorSurf_R42}. On the other hand, 
formulas  \eqref{K_kappa-gh_MinLorSurf_R42} do not change if we  replace $(g_i,h_i)$ with $(g_i,-h_i)$. Hence, the initial quadruple also gives a solution to  system
 \eqref{Nat_Eq_K_kappa_MinLorSurf_R42}.
\end{proof}

Two different quadruples of real functions can  generate according to formulas \eqref{K_kappa-gh_MinLorSurf_R42} one and the same solution to the system of natural equations \eqref{Nat_Eq_K_kappa_MinLorSurf_R42}. 

\begin{thm}\label{Nat_eq_same_K_kappa_hatgh_gh_R42}
Let  $(g_i,h_i)$ and  $(\hat g_i,\hat h_i)$; $i=1, 2$ be  two quadruples of real functions satisfying conditions
 \eqref{CanW_cond_MinLorSurf_R42}. They generate according to formulas \eqref{K_kappa-gh_MinLorSurf_R42} one and the same solution to the system of natural equations \eqref{Nat_Eq_K_kappa_MinLorSurf_R42} if and only if they are related by linear-fractional transformations of the following type:
\begin{gather}\label{hatgh_gh-same_sol_R42}
\hat g_i = \ds\frac{a_1 g_i + b_1}{c_1 g_i + d_1}; \quad 
\hat h_i = \ds\frac{a_2 h_i + b_2}{c_2 h_i + d_2};
\quad 
\begin{array}{l}
a_1,b_1,c_1,d_1,a_2,b_2,c_2,d_2 \in \RR; \\
(a_1d_1-b_1c_1)(a_2d_2-b_2c_2) \neq 0; \  i=1, 2.
\end{array}
\end{gather}
\end{thm}

\begin{proof}
Note that both formulas \eqref{K_kappa-gh_MinLorSurf_R42} and \eqref{hatgh_gh-same_sol_R42} are invariant under substitutions of the form: 
$g_i \rightarrow -g_i$, $h_i \rightarrow -h_i$,  $\hat g_i \rightarrow -\hat g_i$, $\hat h_i \rightarrow -\hat h_i$, the substitutions being simultaneous for $i=1, 2$.
So, it is enough to consider only  the case $g'_1>0$, $h'_1>0$, $\hat g'_1 >0$, $\hat h'_1>0$.

Let  $(g_i,h_i)$ and  $(\hat g_i,\hat h_i)$ generate one and the same solution,  $\M$ and $\hat\M$ be the corresponding surfaces with canonical Weierstrass representation \eqref{CanW_MinLorSurf_R42}. 
We assume that  $\omega_i$ and $\hat\omega_i$ are chosen in such a way that the signs of  $E$ and $\hat E$ coincide with the sign of $\delta$ in the system. We have 
$K=\hat K$ and $\varkappa=\hat\varkappa$.
If $K^2-\varkappa^2>0$, then according to Proposition   \ref{DegP_kind123-K_kappa-tl}, from 
 $g'_1h'_1>0$,  $\hat g'_1 \hat h'_1>0$, and \eqref{alphapp^2_fgh_R42} we get that $\M$ and  $\hat\M$ are of  first type.  
Then, Theorem \ref{Thm-Nat_Eq_K_kappa_R42-tl} implies that $\M$ and $\hat\M$ are related by a proper motion in $\RR^4_2$. Hence, their corresponding null curves are also related by the same proper motion.
So, applying Theorem \ref{W-proper_move_R42} we obtain formulas  \eqref{hatfgh_fgh-proper_move_R42}, i.e. equalities 
 \eqref{hatgh_gh-same_sol_R42} hold true.
Analogously, if $K^2-\varkappa^2<0$ then, $\M$ and $\hat\M$ are of  third type and again  equalities \eqref{hatgh_gh-same_sol_R42} hold true.

Now, let  $(g_i,h_i)$ and  $(\hat g_i,\hat h_i)$ be related by  equalities \eqref{hatgh_gh-same_sol_R42}.
It follows from $g'_1>0$ and $\hat g'_1 >0$ that $a_1d_1-b_1c_1>0$. Analogously,  $a_2d_2-b_2c_2>0$.
We may assume that  $a_1d_1-b_1c_1=1$ and $a_2d_2-b_2c_2=1$, since each linear-fractional transformation can be considered as given by a matrix with  determinant $\pm 1$.
Let us consider the surface $\M$ with canonical Weierstrass representation  \eqref{CanW_MinLorSurf_R42} determined by the given functions $(g_i,h_i)$.
Assume that $\omega_i$ are chosen in such a  way that the sign of  $E$ coincides with the sign of $\delta$.
Let $A$ be a proper orthochronous motion in $\RR^4_2$ generated by the homomorphism  \eqref{Spin_B1B2-A_R42},
where $(B_1,B_2)$ are the matrices of the corresponding linear-fractional transformations. 
We consider the surface  $\hat\M$ obtained by $\M$ under the motion $A$. Then, the null curves corresponding to $\M$ and $\hat\M$ are related by the same motion. 
Therefore,  formulas \eqref{hatfgh_fgh-orthchr_move_R42} with the same linear-fractional transformations are valid for them.  
This means that $\hat\M$ has a canonical Weierstrass representation  \eqref{CanW_MinLorSurf_R42} by the given functions $(\hat g_i,\hat h_i)$.
The two surfaces are congruent by definition and therefore have equal Gauss and normal curvatures for which 
\eqref{K_kappa-gh_MinLorSurf_R42} is valid.
Hence,  $(g_i,h_i)$ and  $(\hat g_i,\hat h_i)$ generate one and the same solution to \eqref{Nat_Eq_K_kappa_MinLorSurf_R42}.
\end{proof}

To get examples of solutions to system  \eqref{Nat_Eq_K_kappa_MinLorSurf_R42} we will consider again the surfaces $\M_1$, $\M_2$, and $\M_3$  defined at the end of Section \ref{W-MinLorSurf_R42}.

\begin{example}\label{exmp-Sol_R42_Enp1}
Let $\M_1$ be the surface defined in Example  \ref{exmp-MinLorSurf_R42_Enp1}.
It is of  first type and its curvatures $K$ and $\varkappa$ expressed in terms of canonical isotropic coordinates $(t_1,t_2)$ are given  by \eqref{exmp-K_kappa-MinLorSurf_kind1_R42}. 
Changing  the parameters according to formulas \eqref{isoterm-isotrop}  we obtain $K$ and $\varkappa$ in terms of canonical isothermal coordinates  $(u,v)$:
\begin{equation}\label{exmp-K_kappa-MinLorSurf_kind1_uv_R42}
K         = \ds\frac{64 u^2 + 576 v^2}{|u^2 - 9 v^2|^3}; \qquad
\varkappa = \ds\frac{-384 u v}{|u^2 - 9 v^2|^3}.
\end{equation}
The last formulas give a solution to  system \eqref{Nat_Eq_K_kappa_MinLorSurf_R42} in the case $\delta=-1$  for which $K^2-\varkappa^2>0$. 

 Note that the surface $\M_2$ of Example \ref{exmp-MinLorSurf_R42_Enp2}  is a surface of second type for which the same formulas 
\eqref{exmp-K_kappa-MinLorSurf_kind1_R42} hold true. Hence, it generates the same solution \eqref{exmp-K_kappa-MinLorSurf_kind1_uv_R42} to system \eqref{Nat_Eq_K_kappa_MinLorSurf_R42}. 
\end{example}

\begin{example}\label{exmp-Sol_R42_Enp3}
Analogously, considering the surface $\M_3$ from Example  \ref{exmp-MinLorSurf_R42_Enp3} which is a surface of third type, we obtain the curvatures:
\begin{equation*}\label{exmp-K_kappa-MinLorSurf_kind3_uv_R42}
K         = \ds\frac{256 u^2 - 384 u v - 256 v^2}{|3 u^2 + 8 u v - 3 v^2|^3}; \qquad
\varkappa = \ds\frac{320 u^2 + 320 v^2}{|3 u^2 + 8 u v - 3 v^2|^3}.
\end{equation*}
The last formulas give a solution to system \eqref{Nat_Eq_K_kappa_MinLorSurf_R42} in the case $\delta=-1$ for which $K^2-\varkappa^2<0$.
\end{example}

\vspace{1.5ex} \textbf{Acknowledgments:}
The  third author is partially supported by the National Science Fund, Ministry of Education and Science of Bulgaria under contract DN 12/2.

\end{document}